\documentclass[12pt,a4paper,oneside]{amsart}
\usepackage[margin=1in]{geometry}
\usepackage{amsmath,amsthm,amssymb}
\usepackage{mathtools}
\usepackage{wasysym}
\usepackage{tabularx}
\usepackage{array}
\newcolumntype{M}[1]{>{\centering\arraybackslash}m{#1}}

\newtheorem{theorem}{Theorem}

\newtheorem{proposition}{Proposition}
\newtheorem{corollary}{Corollary}

\newtheorem*{remark*}{Remark}

\begin{document}
\title{The number of polyiamonds is supermultiplicative}
\author{Vuong Bui}
\address{Vuong Bui, Institut f\"ur Informatik, Freie Universit\"{a}t
Berlin, Takustra{\ss}e~9, 14195 Berlin, Germany, LIRMM, Universit\'e de Montpellier, 161 Rue Ada, 34095 Montpellier, France, and UET, Vietnam National University, Hanoi, 144 Xuan Thuy Street, Hanoi 100000, Vietnam}
\email{bui.vuong@yandex.ru}
\thanks{Part of the work was supported by the Deutsche Forschungsgemeinschaft
(DFG) Graduiertenkolleg ``Facets of Complexity'' (GRK 2434).}

\begin{abstract}
	While the number of polyominoes is known to be supermultiplicative by a simple concatenation argument, it is still unknown whether the same applies to polyiamonds. This article proves that if $\ell,m$ are not both $1$, then $T(\ell+m)\ge T(\ell)T(m)$, for which one can say that the number of polyiamonds $T(n)$ is supermultiplicative. The method is, however, by concatenating, merging and adding cells at the same time. A corollary is an increment of the best known lower bound on the growth constant from $2.8423$ to $2.8578$.
\end{abstract}
\maketitle
\section{Introduction}
A \emph{polyomino} is an edge-connected set of cells on the square lattice. It has become popular in literature, as a theoretical topic with practical applications as well as a recreational medium. Meanwhile, a \emph{polyiamond}, which is our main interest in this article, is an edge-connected set of cells on the triangular lattice. While the cells of the square lattice are identical up to a translation, there are two types of cells for the triangular lattice, as depicted in Fig.~\ref{fig:triangular-lattice}.

\begin{figure}[ht]
	\includegraphics[width=0.25\textwidth]{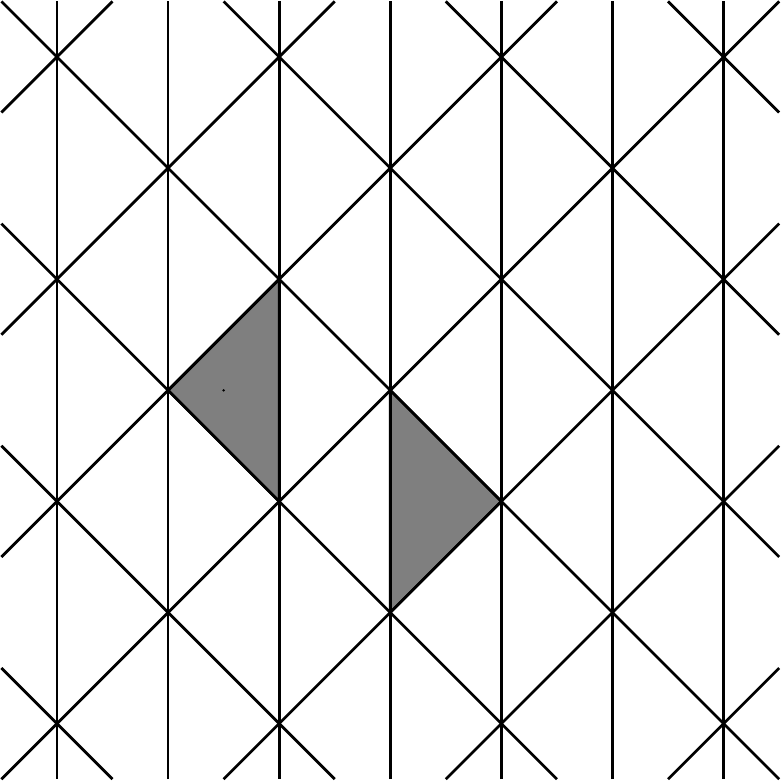}
	\caption{A triangular lattice with two polyiamonds of $1$ cell}
	\label{fig:triangular-lattice}
\end{figure}

Two lattice animals are said to be equivalent if one of them is a translate of the other. On the other hand, we say two animals are distinct if they are not equivalent. We denote by $S(n), T(n)$ the number of all distinct animals with $n$ cells for the square lattice and the triangular lattice, respectively. They are also called the number of \emph{fixed} lattice animals in literature, as there are other variants allowing extra operations, e.g. rotation, reflection, instead of only translation. In the sequel, we always mean fixed lattice animals by lattice animals. 

Some beginning values of $T(n)$ are given in Table \ref{tbl:values}, with the actual polyiamonds for $n=1,2,3$ depicted in Fig.~\ref{fig:polyiamond-counts}. Other values of $T(n)$ for $n\le 75$ can be found in Page $479$ of \cite{jacobsen2009polygons}. The sequence $T(n)$ is also known as the sequence $A001420$ in The On-Line Encyclopedia of Integer Sequences. 

\begin{table}[ht]
	\centering
	\begin{tabular}{ |c|c|c|c|c|c|c|c|c|c|c|c|c|c|c| } 
		\hline 
		$n$ & 1 & 2 & 3 & 4 & 5 & 6 & 7 & 8 & 9 & 10 & 11 & 12 & 13 & 14\\
		\hline
		$T(n)$ & 2 & 3 & 6 & 14 & 36 & 94 & 250 & 675 & 1838 & 5053 & 14016 & 39169 & 110194 & 311751\\
		\hline
	\end{tabular}
	\caption{Some beginning values of $T(n)$}
	\label{tbl:values}
\end{table}

\begin{figure}[ht]
	\includegraphics[width=0.5\textwidth]{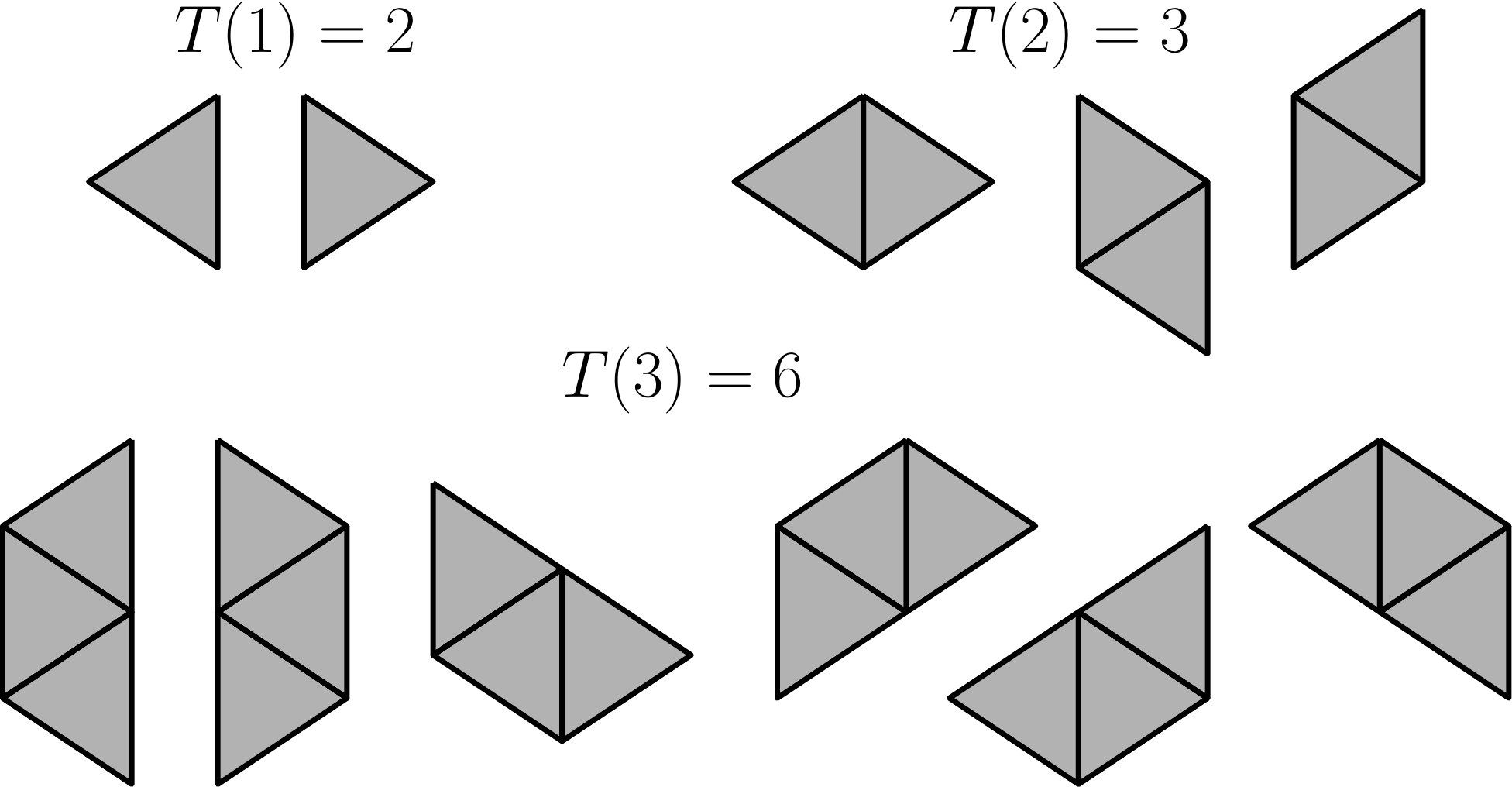}
	\caption{Polyiamonds of $1,2,3$ cells}
	\label{fig:polyiamond-counts}
\end{figure}

\subsection*{Growth constants}
The number of the polyominoes with $n$ cells grows exponentially with a growth constant represented by the limit
\[
	\lambda_S = \lim_{n\to\infty} \sqrt[n]{S(n)}.
\]
The constant is also known as Klarner's constant, since Klarner first gave a proof of the limit by observing that $S(n)$ is supermultiplicative \cite{klarner1967cell}. We remind that a sequence $s_n$ is said to be \emph{supermultiplicative} if $s_{\ell+m}\ge s_\ell s_m$ for any $\ell,m\ge 1$. If $s_n$ is positive and supermultiplicative, then the limit of $\sqrt[n]{s_n}$ exists by Fekete's lemma \cite{fekete1923verteilung}. Note that in principle the limit $\sqrt[n]{s_n}$ can be infinite but in our case we have $S(n)\le 6.75^n$ by a result of Eden \cite{eden1961two} using an encoding technique. In other words, the limit of $\lambda_S$ exists and is finite.

As for the case of polyiamonds, we do not know yet whether $T(n)$ is supermultiplicative. In fact, as far as to the awareness of the author, the only paper that really attempts to prove the limit for polyiamonds is \cite{madras1999pattern}, where it is shown that: Given a lattice with the number $L(n)$ of lattice animals with $n$ cells, we have the growth constant
\[
	\lambda_L = \lim_{n\to\infty} \frac{L(n+1)}{L(n)}.
\]
Note that the result is actually stronger than the existence of the limit $\sqrt[n]{L(n)}$. However, the techniques in use, often known under the name ``pattern theorems'', are quite involved. Therefore, we still desire an easier way to prove the limit
\[
	\lambda_T = \lim_{n\to\infty} \sqrt[n]{T(n)}.
\]
Ideally, it would be something like a concatenation argument as in the case of polyominoes. In fact, we give such a way in Proposition \ref{prop:proving-limit}.

\subsection*{Concatenation arguments and lower bounds}
Before reminding the proof of the supermultiplicativity of $S(n)$ using the concatenation argument, let us define a \emph{lexicographic order} on the cells of a polyomino/polyiamond: Cell $c_1$ is said to be smaller than cell $c_2$ if (i) $c_1$ is on a column to the left of the column of $c_2$, or (ii) both $c_1, c_2$ are on the same column and $c_1$ is below $c_2$. 

With the lexicographic order, we are ready to show the supermultiplicativity of $S(n)$: For every pair of polyominoes $A,B$ with $\ell,m$ cells, respectively, we can give a unique polyomino of $\ell+m$ cells by translating them so that the largest cell of $A$ is adjacent to the left of the smallest cell of $B$. For example, Fig.~\ref{fig:polyomino-concatenation} depicts a case with $A$ in the dark color and $B$ in white.

\begin{figure}[ht]
	\includegraphics[width=0.35\textwidth]{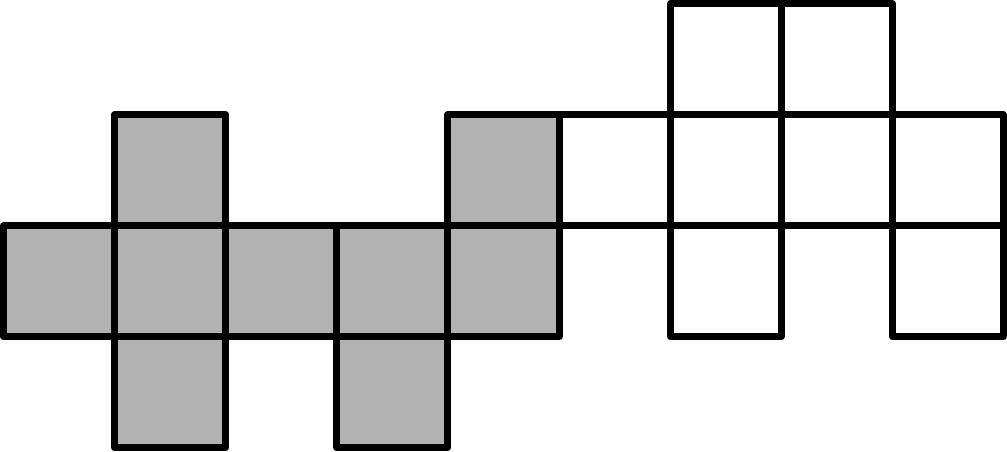}
	\caption{A concatenation of two polyominoes}
	\label{fig:polyomino-concatenation}
\end{figure}

Note that the supermultiplicativity of $S(n)$ does not only give the existence of the limit $\lambda_S$ but also concludes that $\lambda_S=\sup_n \sqrt[n]{S(n)}$ by Fekete's lemma. In other words, if we have the value of $S(n)$ for some $n$, we obtain a lower bound $\sqrt[n]{S(n)}\le \lambda_S$. Note that we do not really need the corollary $\lambda_S=\sup_n \sqrt[n]{S(n)}$ of Fekete's lemma to see that $\sqrt[n]{S(n)}$ is a lower bound. Indeed, we have the infinite sequence $S(n),S(2n), S(3n),\dots$ with the growth rate at least $\sqrt[n]{S(n)}$, since it follows from the supermultiplicativity that $S(kn)\ge S(n)S((k-1)n)\ge S(n)[S(n)]^{k-1} = [S(n)]^k$ for every $k\ge 2$ by induction. In fact, we even have a nondecreasing sequence $\sqrt[n]{S(n)}, \sqrt[2n]{S(2n)}, \sqrt[4n]{S(4n)}, \sqrt[8n]{S(8n)},\dots$ since $\sqrt[2m]{S(2m)}\ge \sqrt[2m]{[S(m)]^2} = \sqrt[m]{S(m)}$ for every $m$.

\subsection*{Difficulties with polyiamonds and previous approaches}
While proving the supermultiplicativity of the number of polyominoes is so straightforward with a simple concatenation argument, we will explain why it is not so simple for the case of polyiamond. At first, we revise the two types of polyiamonds characterized by the type of triangle of the largest cell, as depicted by the two top polyiamonds in Fig.~\ref{fig:polyiamond-types}. Let us denote the types by $\RHD$ and $\LHD$, with respect to the direction of the triangle. We also denote the number of polyiamonds of these types by $T_\RHD(n)$ and $T_\LHD(n)$, respectively. If we characterize by the smallest cells, we have the two bottom polyiamonds in Fig.~\ref{fig:polyiamond-types}. We denote the types by $\lhd$ and $\rhd$, and the numbers of polyiamonds of them by $T_\lhd(n), T_\rhd(n)$, in the same manner. The article \cite{barequet2019improved} observes the following relations:
\[
	T_\RHD(n)=T_\lhd(n),\qquad T_\LHD(n)=T_\rhd(n),
\]
as a polyiamond of type $\RHD$ is a reflex of a polyiamond of type $\lhd$, and likewise for the pair $\LHD,\rhd$. Also, we have
\[
	T_\LHD(n)=T_\RHD(n-1),
\]
because the cell $\LHD$, which is the largest cell, needs another cell $\RHD$ just below it to connect to the remaining cells. 

\begin{figure}[ht]
	\includegraphics[width=0.3\textwidth]{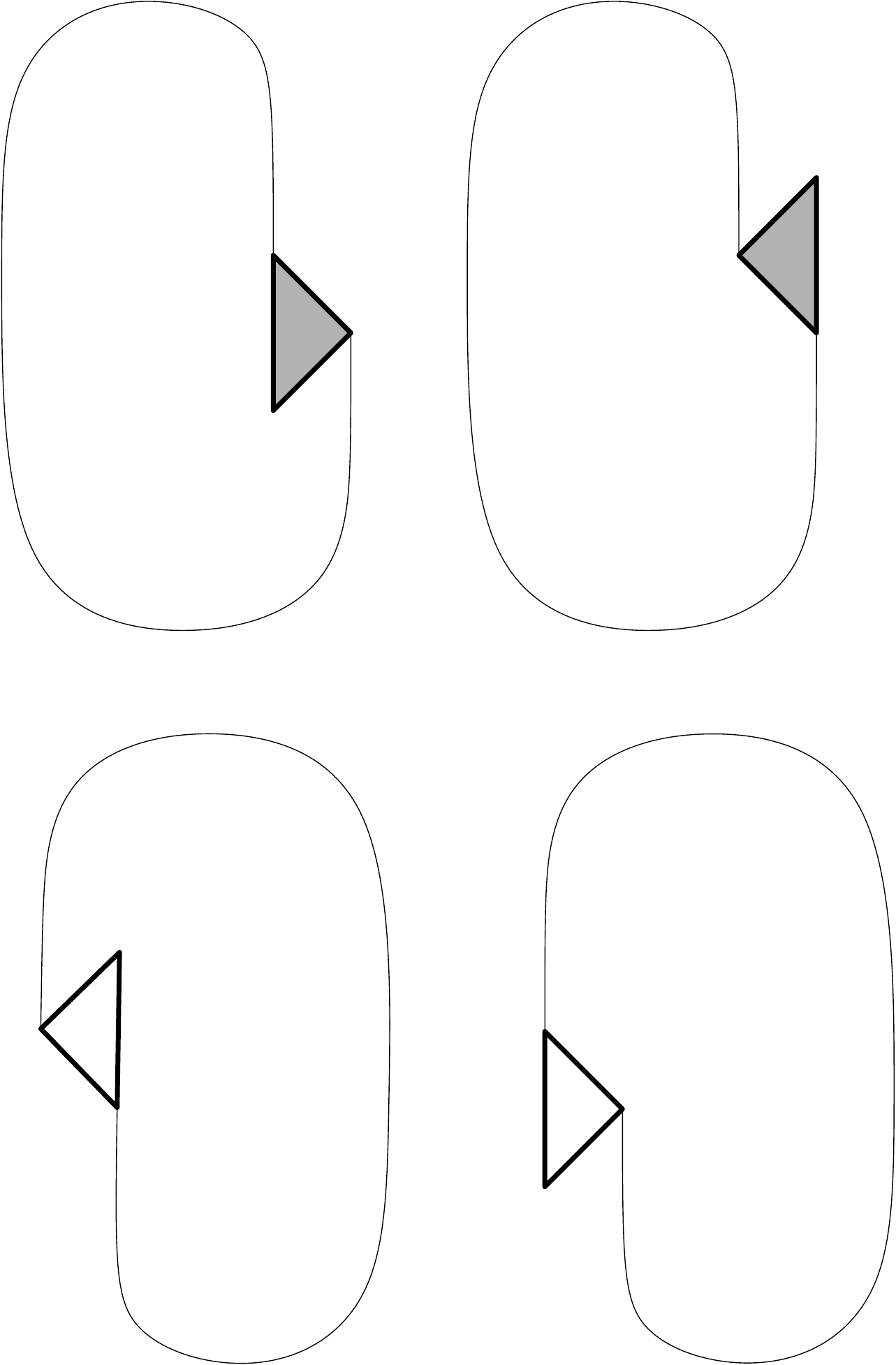}
	\caption{Types of polyiamonds by the largest/smallest cells}
	\label{fig:polyiamond-types}
\end{figure}

Also in \cite{barequet2019improved}, we observe that a polyiamond of type $\RHD$ can be concatenated (by a similar strategy to polyominoes) to a polyiamond of type $\lhd$ only, and a polyiamond of type $\LHD$ can be concatenated to a polyiamond of type $\rhd$ only. The former can be done in only one way, but we allow the latter to be done in two ways, as depicted in Fig.~\ref{fig:polyiamond-concatenations}.

\begin{figure}[ht]
	\includegraphics[width=0.4\textwidth]{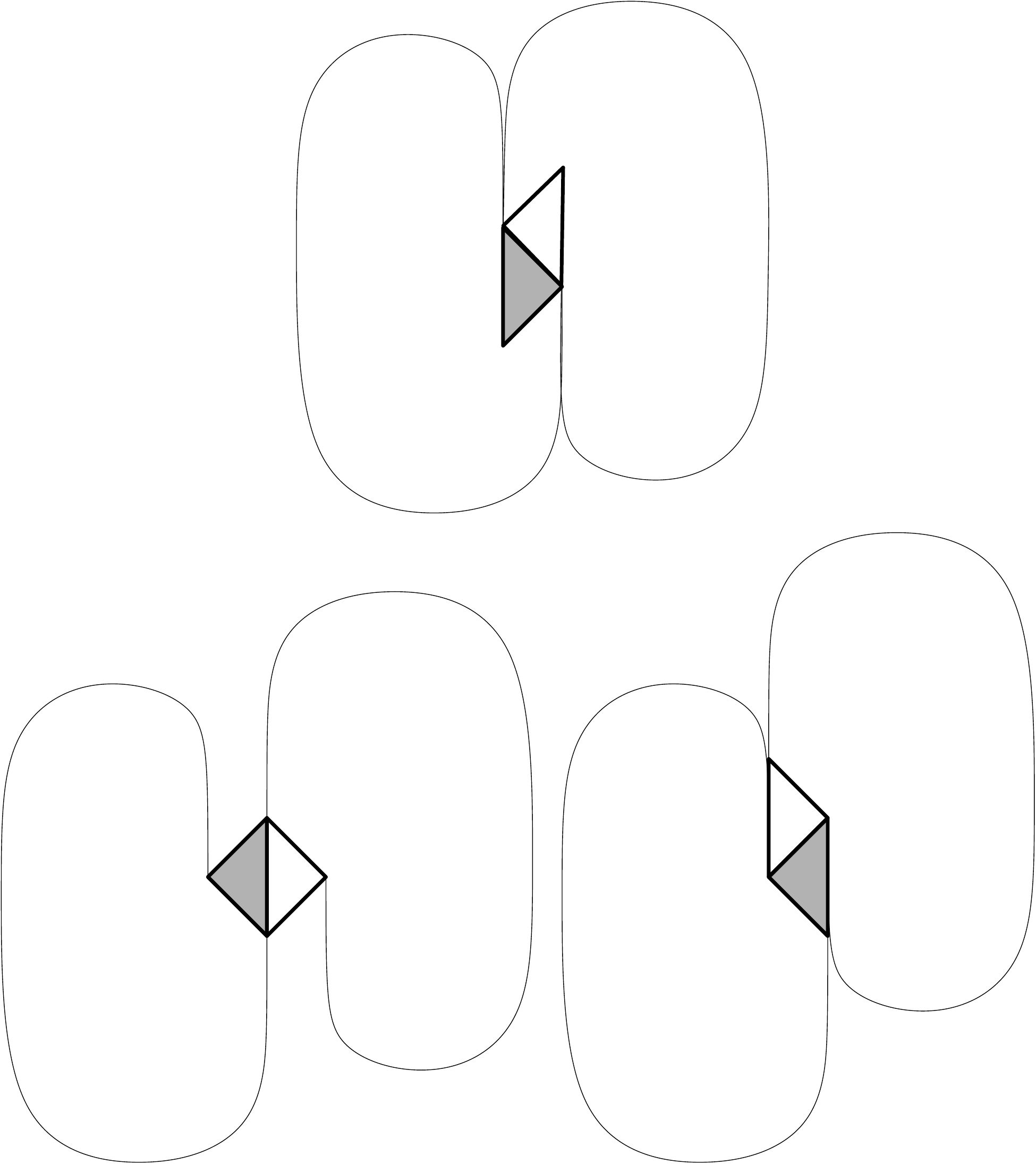}
	\caption{Concatenations of polyiamonds}
	\label{fig:polyiamond-concatenations}
\end{figure}

Although not every pair of polyiamonds can be concatenated in the traditional way, a concatenation argument in \cite{barequet2019improved} gives the following result.
\begin{proposition}[Barequet, Shalah, and Zheng 2019]
	For every $n$,
	\begin{equation} \label{eq:known-bound-polyiamond}
		T(2n)\ge \frac{2}{3} [T(n)]^2.
	\end{equation} 
\end{proposition}

Let us quickly sketch the argument in \cite{barequet2019improved}. 
\begin{proof}
	Since all the three resulting polyiamonds in Fig.~\ref{fig:polyiamond-concatenations} are distinguishable, it follows that for every $n$,
	\[
		T(2n)\ge T_\RHD(n)T_\lhd(n) + 2T_\LHD(n)T_\rhd(n).
	\]
	As $T_\RHD(n)=T_\lhd(n), T_\LHD(n)=T_\rhd(n)$ and $T(n)=T_\RHD(n)+T_\LHD(n)$, we have
	\begin{align*}
		T(2n)&\ge [T_\RHD(n)]^2 + 2[T_\LHD(n)]^2 \\
		&= \frac{2}{3} [T_\RHD(n)]^2 + \frac{2}{3}[T_\LHD(n)]^2 + \frac{1}{3} [T_\RHD(n)]^2 + \frac{4}{3}[T_\LHD(n)]^2 \\
		&\ge \frac{2}{3} [T_\RHD(n)]^2 + \frac{2}{3}[T_\LHD(n)]^2 + 2\sqrt{\frac{1}{3}} T_\RHD(n) \sqrt{\frac{4}{3}} T_\LHD(n) \\
		&= \frac{2}{3} [T_\RHD(n) + T_\LHD(n)]^2\\
		&= \frac{2}{3} [T(n)]^2.\qedhere
	\end{align*}
\end{proof}

Relation \eqref{eq:known-bound-polyiamond} gives a lower bound for the growth constant $\lambda_T$ in terms of $T(n)$. Indeed, we have $\frac{2}{3} T(2n) \ge \left[\frac{2}{3} T(n)\right]^2$, i.e. $\sqrt[2n]{\frac{2}{3} T(2n)} \ge \sqrt[n]{\frac{2}{3} T(n)}$ for every $n$. In other words, the sequence $\left\{\sqrt[2^t\cdot n]{\frac{2}{3} T(2^t\cdot n)}\right\}_t$ is nondecreasing. It follows that $\lambda_T\ge \sqrt[n]{\frac{2}{3} T(n)}$ for every $n$. The actual values of $T(n)$, which is the sequence $A001420$ in The On-Line Encyclopedia of Integer Sequences, are known for all $n\le 75$ (see Page $479$ of \cite{jacobsen2009polygons}) with 
\[
	T(75)=15936363137225733301433441827683823.
\]
It gives the following bound in \cite{barequet2019improved}:\footnote{In \cite{barequet2019improved}, the lower bound is actually $2.8424$, as an approximation of $\sqrt[75]{\frac{2}{3} T(75)}$. But Barequet once hinted in a private communication that one would better use a strict lower bound than an approximation for the sake of rigor.}

\begin{equation} \label{eq:current-best}
	\lambda_T\ge \sqrt[75]{\frac{2}{3} T(75)} > 2.8423.
\end{equation}

In the same work \cite{barequet2019improved}, by assuming an unproved conjecture,\footnote{It is actually called an assumption in \cite{barequet2019improved}: For every $n$, we have $T(n+1)/T(n)\le 4$.} there is a stronger but conditional relation: \begin{equation} \label{eq:unproved}
	T(2n)\ge 0.7122 \, [T(n)]^2,
\end{equation}
which gives the corresponding tentative bound
\[
    \lambda_T\ge \sqrt[75]{0.7122 \cdot T(75)} > 2.8448.
\]

In the following section, we will present a simple proof that the limit of the growth constant exists and a stronger result on $T(n)$ that it is supermultiplicative.

\section{Main theorem}
\subsection*{A simple proof of the limit of the growth constant for polyiamonds}
Neither \eqref{eq:known-bound-polyiamond} nor \eqref{eq:unproved} is enough to prove the limit of $\sqrt[n]{T(n)}$, e.g. the sequence $T'(n)$ with $T'(2k)=3^k$ and $T'(2k+1)=2^{2k+1}$ (for any $k$) satisfies the relations in \eqref{eq:known-bound-polyiamond} and \eqref{eq:unproved} but we do not have the limit of $\sqrt[n]{T'(n)}$. However, we can adapt the argument in \cite{barequet2019improved} with little extra effort to make it more general (by not fixing $n$ but letting $\ell,m$ arbitrary) but weaker (by a constant $\frac{1}{3}$ instead of $\frac{2}{3}$) as in the following proposition.
\begin{proposition}\label{prop:proving-limit}
	For every $\ell,m$,
	\begin{equation}\label{eq:one-third}
		T(\ell+m)\ge \frac{1}{3} T(\ell)T(m).
	\end{equation}
\end{proposition}
\begin{proof}
	We first observe that $T_\RHD(n)$ is a nondecreasing sequence, as $T_\RHD(n)\le T_\RHD(n+1)$ can be seen by an injection map from a polyiamond of $n$ cells to a polyiamond of $n+1$ cells (both of type $\RHD$): We simply add a cell just below the smallest cell of the polyiamond of $n$ cells. It follows that there are at least as many polyiamonds of $n$ cells of type $\RHD$ as those of type $\LHD$, since $T_\RHD(n)\ge T_\RHD(n-1) = T_\LHD(n)$ (when $n-1=0$, it is trivial). For every $\ell,m$, we have\footnote{In fact, we do not use two concatenations of two polyiamonds of types $\LHD$ and $\rhd$ as in Fig.~\ref{fig:polyiamond-concatenations} but one only is sufficient, i.e. we can start with $T(\ell+m) \ge T_\RHD(\ell)T_\RHD(m) + T_\LHD(\ell)T_\LHD(m)$. That is why we come up with the weaker constant $\frac{1}{3}$.}
	\begin{align*}
		3T(\ell+m) &\ge 3[T_\RHD(\ell)T_\RHD(m) + 2T_\LHD(\ell)T_\LHD(m)] \\
		&\ge T_\RHD(\ell)T_\RHD(m) + T_\RHD(\ell)T_\LHD(m) + T_\LHD(\ell)T_\RHD(m) + T_\LHD(\ell)T_\LHD(m) \\
		&= [T_\RHD(\ell)+T_\LHD(\ell)][T_\RHD(m)+T_\LHD(m)] \\
		&= T(\ell)T(m).\qedhere
	\end{align*}
\end{proof}
Although the constant $\frac{1}{3}$ of \eqref{eq:one-third} is weaker than the constant $\frac{2}{3}$ of \eqref{eq:known-bound-polyiamond} in bounding the growth constant, the former gives a proof of the limit of $\sqrt[n]{T(n)}$ (the sequence $\{\frac{1}{3} T(n)\}_n$ is supermultiplicative). To the awareness of the author, there is no other such simple proof of the limit of $\sqrt[n]{T(n)}$, even though the manipulations are so straightforward. (The only known proof is by ``pattern theorems'' in \cite{madras1999pattern}, which is fairly involved.)

\begin{corollary}
	The limit $\lambda_T=\sqrt[n]{T(n)}$ exists.
\end{corollary}

Note that the fact that $\lambda_T$ is finite is already given in \cite{lunnon1972counting} with $\lambda_T\le 4$, using the encoding technique that was introduced by Eden \cite{eden1961two} to prove $\lambda_S\le 6.75$. 

\subsection*{The number of polyiamonds is  supermultiplicative}
It is in fact unnecessary to sacrifice the constant $\frac{2}{3}$ to obtain such a general form in Proposition \ref{prop:proving-limit}. Quite the contrary, we can improve the constant to $1$.
Indeed, while the proofs of the relations use concatenation arguments only, we can show that $T(n)$ is actually supermultiplicative by concatenating two polyiamonds together with merging and adding some certain cells at the same time. This even improves the unproved relation \eqref{eq:unproved}.

\begin{theorem} \label{thm:polyiamond}
	The number of polyiamonds is supermultiplicative, in the sense that if $\ell,m$ are not both $1$, then $T(\ell+m)\ge T(\ell)T(m)$.
\end{theorem}
Note that $T(1)T(1)=2\cdot 2=4$ while $T(2)=3$. However, this is the only exception, for which we may still conclude that $T(n)$ is supermultiplicative.

As Fekete's lemma still works for this type of supermultiplicativity, the corresponding lower bound is $\lambda_T\ge \sqrt[n]{T(n)}$, and when $n=75$, we have
\[
	\lambda_T\ge \sqrt[75]{T(75)} > 2.8578,
\]
which improves the current best lower bound $2.8423$ in \eqref{eq:current-best}.

\begin{proof}
	As one of $\ell, m$ is at least $2$, we let $m\ge 2$. For any such pair of $\ell,m$, we construct an injective map such that for any two polyiamonds $A, B$ of respectively $\ell, m$ cells, the map gives a unique polyiamond $C$ of $n = \ell + m$ cells. This is sufficient for showing that $T(\ell+m) \ge T(\ell) T(m)$.

	Let us first consider the following case (denoted by Case 0) where we can easily deal with: The rightmost column of $A$ has a cell $\LHD$ and the leftmost column of $B$ has a cell $\rhd$. In this case, we translate $A$ and $B$ such that the largest $\LHD$ of $A$ touches the smallest $\rhd$ of $B$ on the left. This gives a connected polyiamond $C$ of $n$ cells. Fig.~\ref{fig:case0} depicts the case with the largest $\LHD$ of $A$ in the dark color and the smallest $\rhd$ of $B$ in white. Note that some leftmost columns of $C$ have $\ell$ cells while the remaining columns have $m$ cells. 

	\begin{figure}[ht]
		\includegraphics[width=0.25\textwidth]{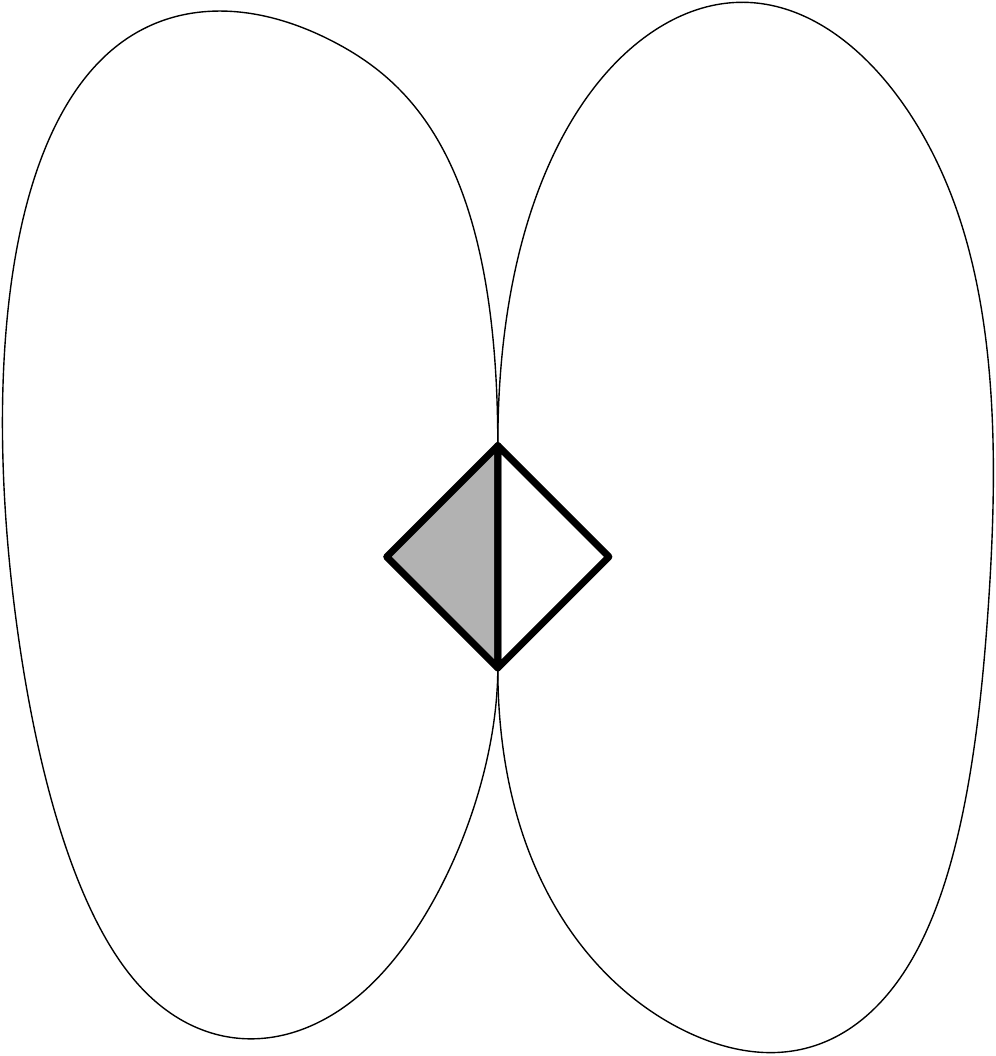}
		\caption{Case 0}
		\label{fig:case0}
	\end{figure}

	Before considering the three remaining cases where we assume that we \emph{do not have} Case 0, the reader may quickly check Table \ref{tbl:cases} of the properties that allow us to distinguish the resulting polyiamonds $C$.\footnote{An entry ``any'' means the value can be any, and depends on particular pairs of polyiamonds.} 
	\begin{table} 
		\begin{tabular}{|M{0.1\textwidth}|M{0.15\textwidth}|M{0.25\textwidth}|M{0.25\textwidth}|} 
			\hline
			Case & Type of the $\ell$-th cell & Some leftmost columns with precisely $\ell$ cells & $2$ adjacent cells immediately on top of the $\ell$-th cell \\ 
			\hline
			0 & any & Yes & No \\ 
			\hline
			1 & $\RHD$ & No & any \\ 
			\hline
			2 & $\LHD$ & No & Yes \\ 
			\hline
			3 & $\LHD$ & No & No \\ 
			\hline
		\end{tabular}
		\caption{Cases and properties}
		\label{tbl:cases}
	\end{table}
	Denoting by $P(A)$ the largest cell of $A$ and by $Q(B)$ the smallest cell of $B$, we consider the three remaining cases:
	\begin{enumerate}
		\item\label{case_a}
			If $P(A) = \RHD$ and $Q(B) = \lhd$, we translate $A, B$ such that $P(A)$ is just below $Q(B)$. This gives a connected polyiamond $C$ of $n$ cells, whose $\ell$-th cell is $\RHD$, and there are no leftmost columns of $C$ with precisely $\ell$ cells. Fig.~\ref{fig:case1} depicts this case, which is actually an extract of Fig.~\ref{fig:polyiamond-concatenations}.

			\begin{figure}[ht]
				\includegraphics[width=0.25\textwidth]{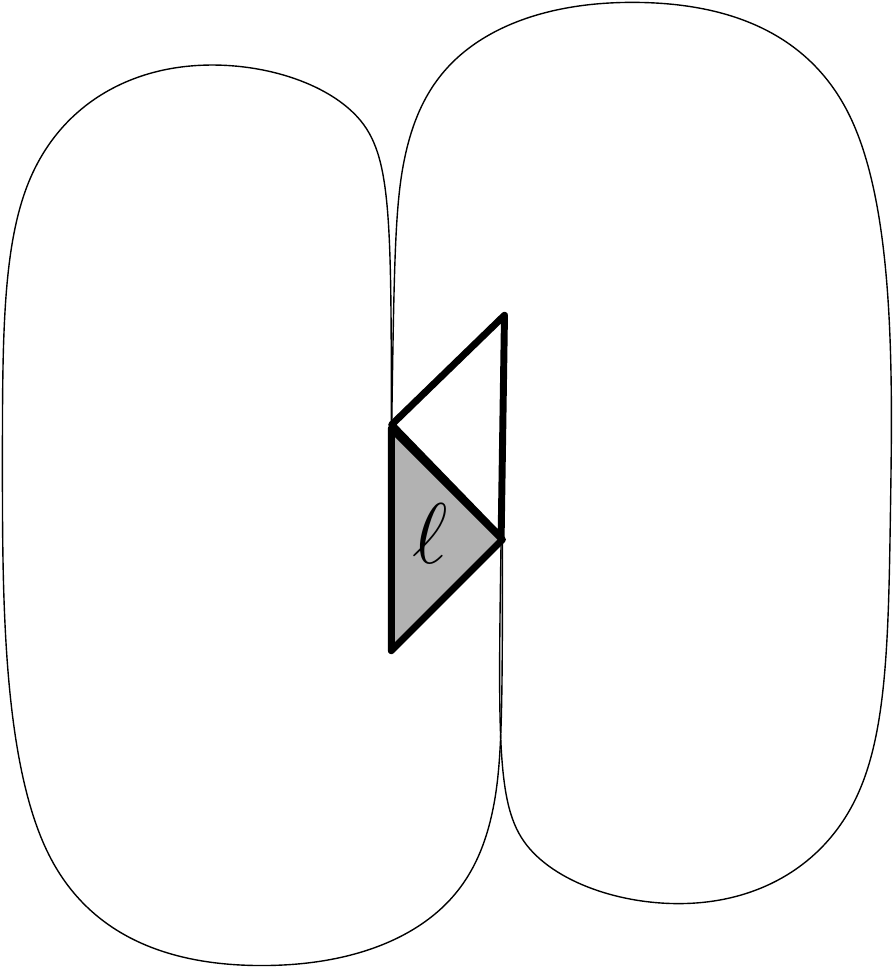}
				\caption{Case 1}
				\label{fig:case1}
			\end{figure}

		\item\label{case_b}
			If $P(A) = \RHD$ and $Q(B) = \rhd$, we translate $A, B$ such that $P(A)$ coincides with $Q(B)$. As we do not have Case 0, there is no cell $\LHD$ on the rightmost column of $A$, hence we can add a cell $\LHD$ just below the position of $P(A)$ (and now also of $Q(B)$) in order to have a polyiamond $C$ of $n$ cells. Fig.~\ref{fig:case2} depicts the case with $P(A)$ and $P(B)$ both situating at the stroked cell, which is the $(\ell+1)$-st cell of $C$. 
            On top of $Q(B)=\rhd$, there is another cell $\lhd$ because the cell $\rhd$ needs a cell $\lhd$ just on top of it to get connected to the rest of $B$ (note that $m>2$ and this is the only place we need the assumption). Meanwhile, the newly added cell $\LHD$ is the $\ell$-th cell of $C$, whose position is just below $P(A)=Q(B)$. In other words, there are at least two adjacent cells immediately on top of the $\ell$-th cell.

			\begin{figure}[ht]
				\includegraphics[width=0.25\textwidth]{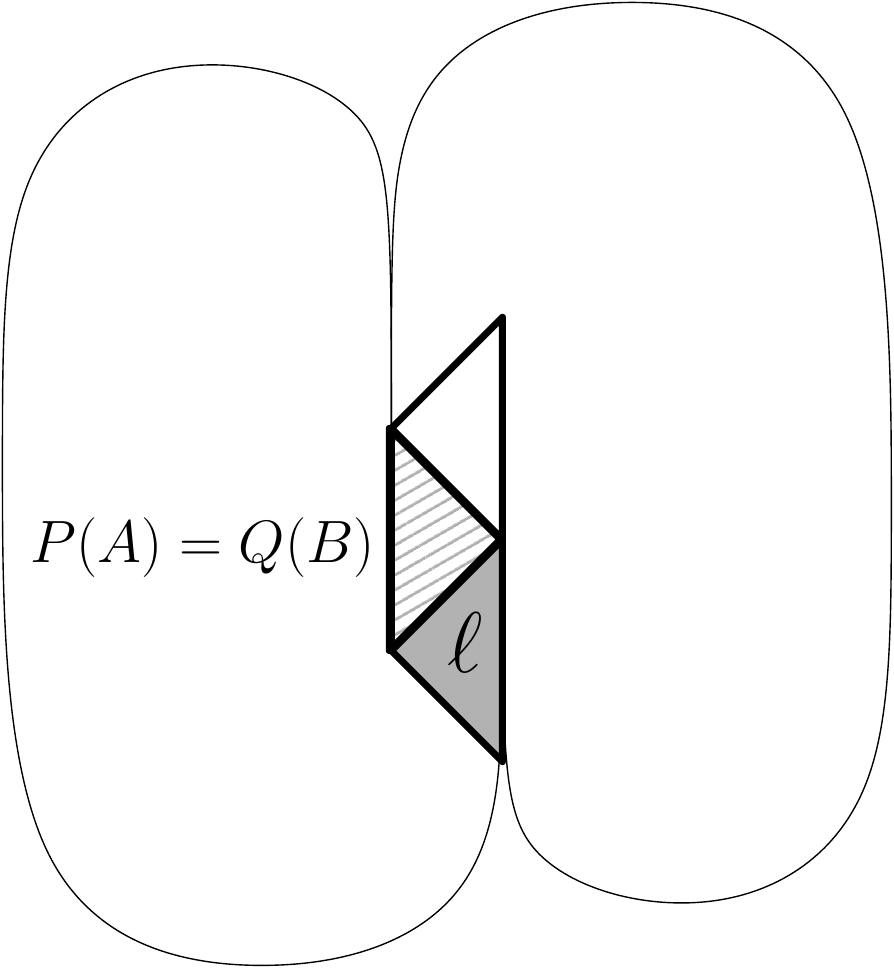}
				\caption{Case 2}
				\label{fig:case2}
			\end{figure}

		\item\label{case_c}
			If $P(A) = \LHD$ and $Q(B) = \lhd$, we translate $A, B$ such that $P(A)$ coincides with $Q(B)$. After that, we add a cell just on top of the largest cell in the leftmost column of $B$. This gives a connected polyiamond $C$ of $n$ cells, whose $\ell$-th cell is a cell $\LHD$ (originally $P(A)$ and $Q(B)$ before merging). 
   
            Fig.~\ref{fig:case3} depicts the two cases: (i) there are more than one cell on the column of $Q(B)$ in $B$ and (ii) $Q(B)$ is the only cell in its column in $B$. In the case (i), the largest cell in the leftmost column of $B$ is depicted as a cell $\lhd$ in white color. In both cases, $P(A)$ and $Q(B)$ both situate at the stroked cell, and the newly added cell $\RHD$ is in black color. 
            
            We remark that there are no two adjacent cells immediately on top of the $\ell$-th cell of $C$. It is because either the position just on top of the $\ell$-th cell of $C$, which is formerly $Q(B)$, is empty in the case (i) (due to not having Case 0), or there is only one cell above the $\ell$-th cell of $C$ in the case (ii). (In the case (i), the two top-most cells in the column of the $\ell$-th cells of $C$ are above the $\ell$-th cell but not \emph{immediately on top} of the $\ell$-th cell.)
			\begin{figure}[ht]
				\includegraphics[width=0.6\textwidth]{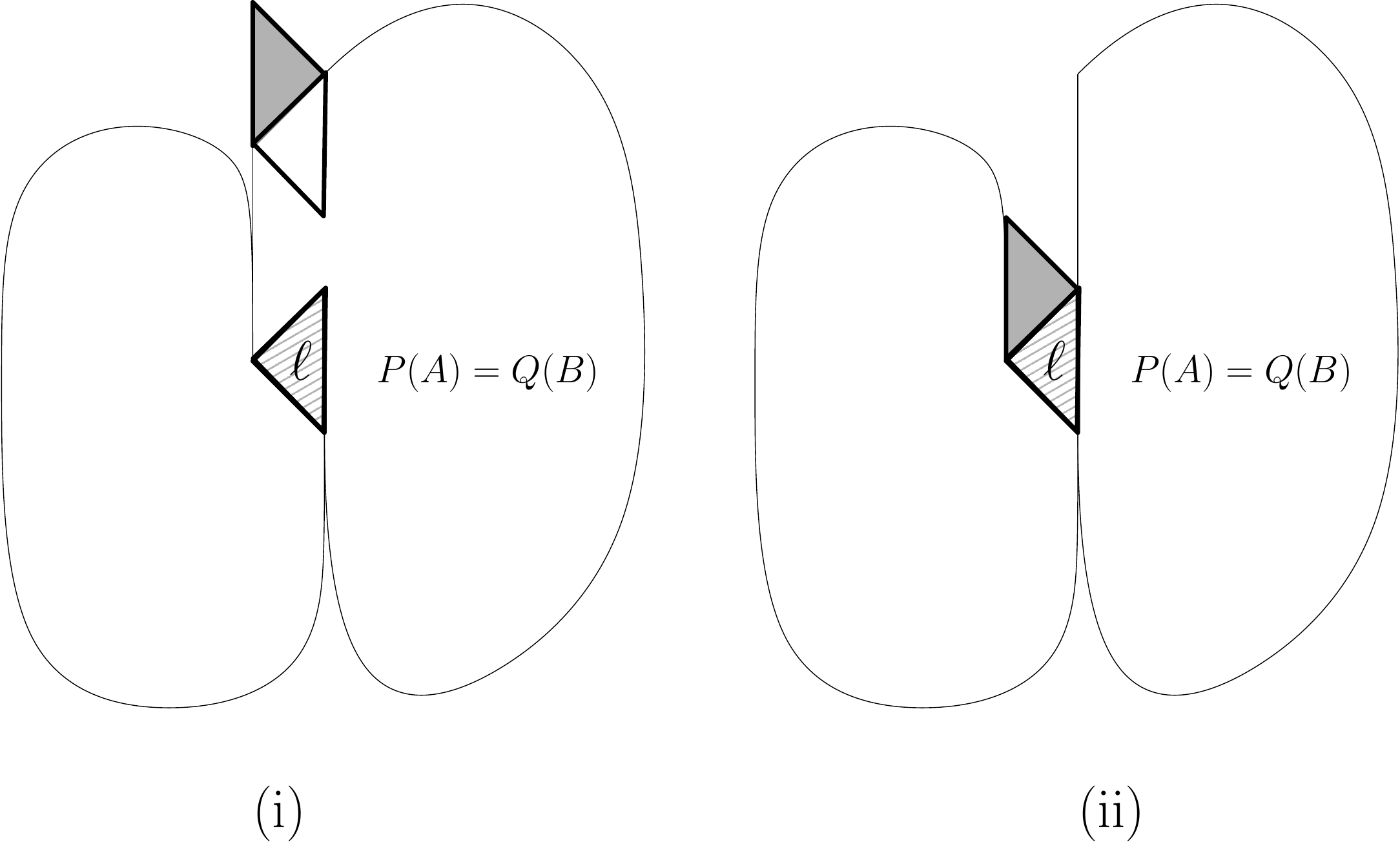}
				\caption{Case 3}
				\label{fig:case3}
			\end{figure}
	\end{enumerate}

	All the cases were covered, and we can distinguish them by Table \ref{tbl:cases}. Therefore, the number of polyiamonds is supermultiplicative in the sense stated in the conclusion.
\end{proof}

\subsection*{Dual representation with the supermultiplicativity in the original sense}
We do not have the supermultiplicativity in the original sense for $T(n)$ since $T(1)=2$ makes $T(1)T(1)>T(2)$ (the only exception). However, we show that in the dual representation the value of $T(1)$ would be $1$ while all other $T(n)$ for $n\ge 2$ remain the same values.

At first, we can see a polyomino in a different way: A polyomino is a set of points in the square lattice so that the induced graph by the points is connected. This dual representation is illustrated in Fig.~\ref{fig:polyomino-lattice}.

\begin{figure}[ht]
	\includegraphics[width=0.5\textwidth]{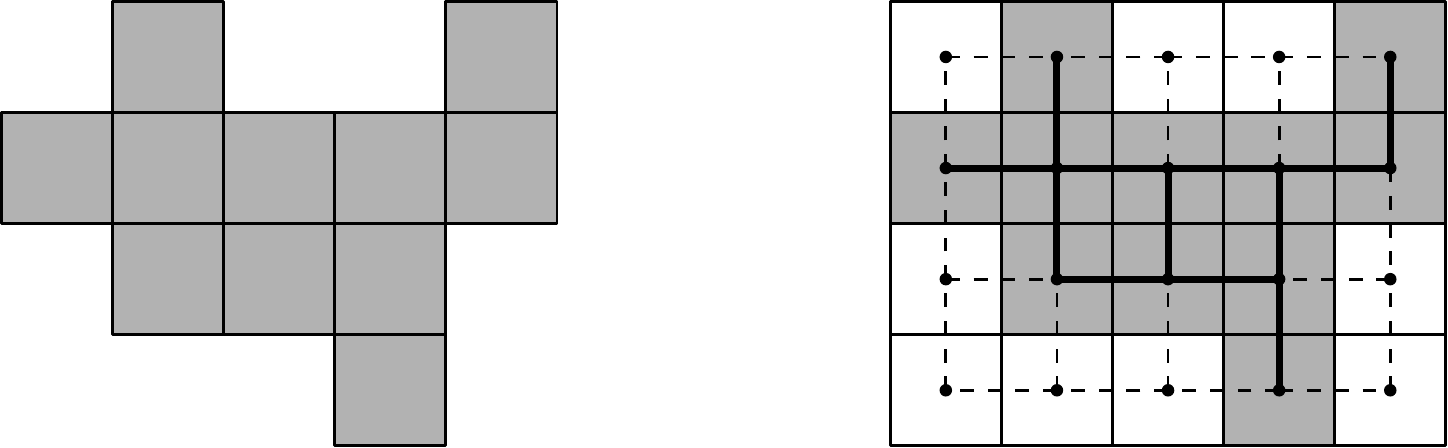}
	\caption{A polyomino with its dual presentation}
	\label{fig:polyomino-lattice}
\end{figure}

The situation for triangular lattice is a bit different. The dual representation of the triangular lattice is the hexagonal lattice (the honeycomb lattice): Each triangle cell in the triangular lattice corresponds to a point in the hexagonal lattice and the edge connectedness of the triangle cells are presented by the edges connecting the points in the hexagonal lattice. The dual representation is illustrated in Fig.~\ref{fig:polyiamond-lattice}.

\begin{figure}[ht]
	\includegraphics[width=0.4\textwidth]{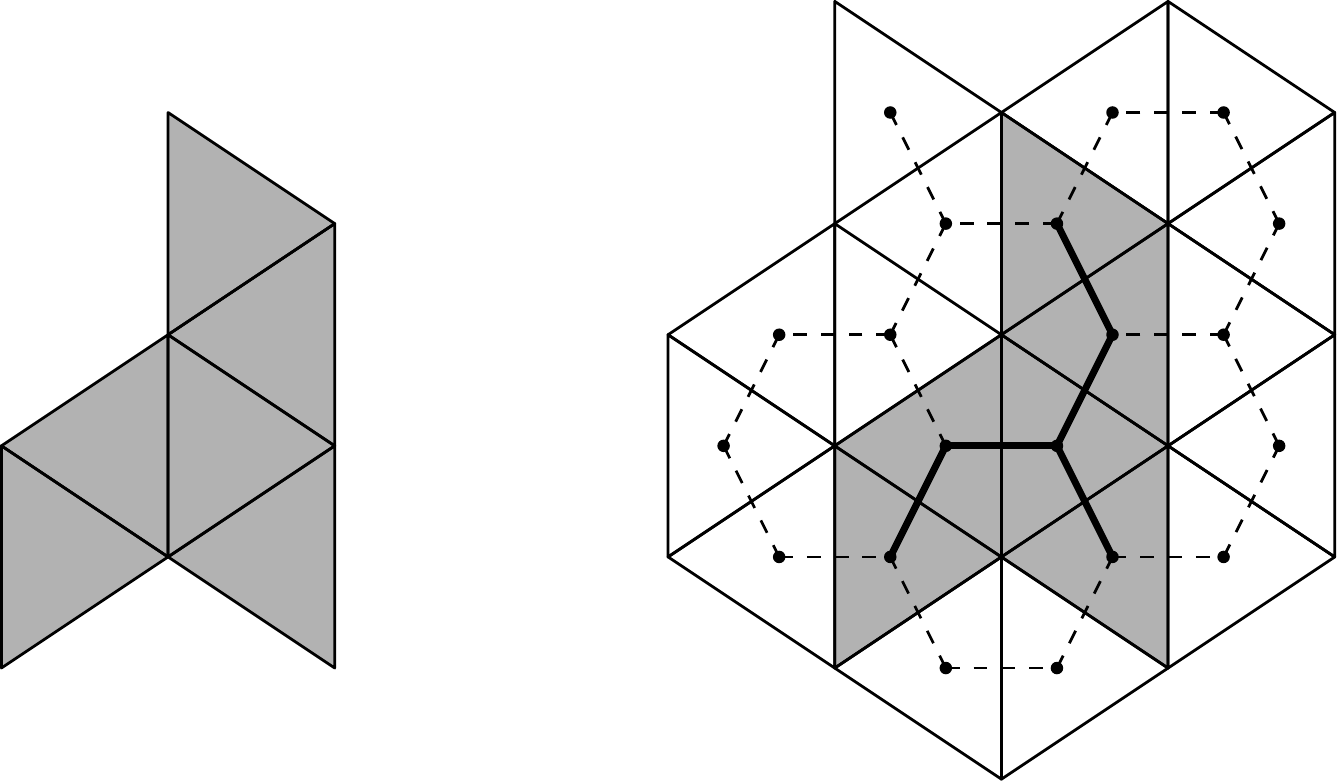}
	\caption{A polyiamond with its dual presentation in the hexagonal lattice}
	\label{fig:polyiamond-lattice}
\end{figure}

We would remark here that the choice of representation affects the value of $T(1)$.
In the original representation by the cells in the triangular lattice, there are two polyiamonds of one cell with opposite directions. In the representation by the points in the hexagonal lattice, the number would be one, as every point is a translate of any other point.
However, both representations give the same value of $T(n)$ for any $n\ge 2$. The reason is that any edge in the hexagonal lattice determines uniquely the types of the corresponding cells of the endpoints in the triangular lattice and the positions of the cells. Indeed, there are three classes of edges (which are unique up to a translation): upward, horizontal, downward (with respect to the direction from left to right). The corresponding cells with the positions are: a $\RHD$ is below a $\LHD$ (upward), a $\LHD$ is to the left of a $\RHD$ (horizontal), a $\RHD$ is above a $\LHD$ (downward).

\subsection*{Acknowledgments}
The author would like to thank G\"unter Rote and Gill Barequet for some interesting discussions and their helpful comments on an early draft of this paper.

\bibliographystyle{unsrt}
\bibliography{polyia-super}

\end{document}